\documentclass[11pt]{amsart}
\usepackage{graphicx}
\usepackage{color}

\newtheorem{thm}{Theorem}[section]

\newtheorem{lem}[thm]{Lemma}
\newtheorem{prop}[thm]{Proposition}
\theoremstyle{definition}
\newtheorem{defn}[thm]{Definition}
\theoremstyle{remark}
\newtheorem{rem}[thm]{Remark}
\newtheorem{ej}[thm]{Example}
\numberwithin{equation}{section}

\usepackage{color}

\newcommand{\Pp}{\mathcal{P}}
\newcommand{\Ii}{\mathcal{I}}
\newcommand{\Rr}{\mathcal{R}}
\newcommand{\Ind}{\operatorname{Ind}}
\newcommand{\End}{\operatorname{End}}
\newcommand{\Hom}{\operatorname{Hom}}
\newcommand{\Ext}{\operatorname{Ext}}
\newcommand{\Der}{\operatorname{D^b}}
\newcommand{\gldim}{\operatorname{gldim}}

\newcommand{\ZZ}{\mathbb{Z}}

\def\semidirprod{\begin{picture}(8,8)\qbezier(2,0.5)(5,3.5)(8,6.5)\qbezier(2,6.5)(5,3.5)(8,0.5)\put(2,0.5){\line(0,1){6}}\end{picture}}

\newcommand{\HVCenter}[1]{\setbox 0=\hbox{#1}%
        \dimen0=\wd0%
        \dimen1=\ht0%
        \divide\dimen0 by 2%
        \divide\dimen1 by 2%
        \hskip -\dimen0%
        \lower \dimen1%
        \box0%
        \hskip -\dimen0}
\newcommand{\HBCenter}[1]{\setbox 0=\hbox{#1}%
        \dimen0=\wd0%
        \dimen1=\ht0%
        \divide\dimen0 by 2%
        \hskip -\dimen0%
        \box0%
        \hskip -\dimen0}
\newcommand{\LTCenter}[1]{\setbox 0=\hbox{#1}%
        \dimen1=\ht0%
        \lower \dimen1%
        \box0%
        \hskip -\dimen0}
\newcommand{\HTCenter}[1]{\setbox 0=\hbox{#1}%
        \dimen0=\wd0%
        \dimen1=\ht0%
        \divide\dimen0 by 2%
        \hskip -\dimen0%
        \lower \dimen1%
        \box0%
        \hskip -\dimen0}
\newcommand{\RTCenter}[1]{\setbox 0=\hbox{#1}%
        \dimen0=\wd0%
        \dimen1=\ht0%
        \hskip -\dimen0%
        \lower \dimen1%
        \box0%
        \hskip -\dimen0}
\newcommand{\RBCenter}[1]{\setbox 0=\hbox{#1}%
        \dimen0=\wd0%
        \dimen1=\ht0%
        \hskip -\dimen0%
        \box0%
        \hskip -\dimen0}
\newcommand{\RVCenter}[1]{\setbox 0=\hbox{#1}%
        \dimen0=\wd0%
        \dimen1=\ht0%
        \divide\dimen1 by 2%
        \hskip -\dimen0%
        \lower \dimen1%
        \box0%
        \hskip -\dimen0}
\newcommand{\LVCenter}[1]{\setbox 0=\hbox{#1}%
        \dimen1=\ht0%
        \divide\dimen1 by 2%
        \lower \dimen1%
        \box0%
        \hskip -\dimen0}

\newcommand{\mylabel}[1]{\put(0,0){\color{white} \circle*{7}}\put(0,0){\circle{7}}\put(0,0){\HVCenter{\tiny\bf #1}}}
\newcommand{\mylabeltwo}[2]{\put(0,0){\color{white} \circle*{7}}\put(0,0){\circle{7}}\put(0,0){\HVCenter{\tiny\bf #1}}\put(0,4.5){\HVCenter{\tiny #2}}}

\newcommand{\mypdott}{\put(0,0){\circle*{2}}}

\begin{document}

\title[On ${\bf m}$-Cluster tilted algebras and trivial extensions]{On ${\bf m}$-Cluster tilted algebras and trivial extensions}%
\author[Fern\'andez]{Elsa Fern\'andez}

\address{Facultad de Ingenier\'{\i}a, Universidad Nacional de la Patagonia San Juan Bosco, 9120 Puerto Madryn, Argentina
}%
\email{elsafer9@gmail.com}%
\author[Pratti]{Isabel Pratti}
\address{Facultad de Ciencias Exactas y Naturales, Universidad Nacional de Mar del Plata, 7600 Mar del Plata, Argentina}%
\email{nilprat@mdp.edu.ar}%

\author[Trepode]{Sonia Trepode}
\address{Facultad de Ciencias Exactas y Naturales, Universidad Nacional de Mar del Plata, 7600 Mar del Plata, Argentina }%
\email{strepode@mdp.edu.ar}%

\thanks{
The third author is a researcher from CONICET, Argentina. The
authors acknowledge partial support from CONICET and ANPCyT.
The authors also would like to thank Mar\'ia In\'es Platzeck
for useful discussion and  suggestions.}

\subjclass[2000]{
16G70,  16G70, 16E10}
\keywords{}%

\begin{abstract}In this work we study the connection between
iterated tilted algebras and m-cluster tilted algebras. We show that
 an iterated tilted algebra induces an m-cluster tilted
algebra. This m-cluster tilted algebra can be seen as a trivial extension of another
iterated tilted algebra which is derived equivalent to the original one. We give a procedure to find
this new iterated tilted algebra. These m-cluster tilted algebras are quotients of higher relation extensions.
\end{abstract}
\maketitle

{\bf Introduction.}
Let $H$ be a finite dimensional hereditary algebra. The cluster
category $C = {\mathcal{C}}_H$ associated with $H$ was defined and
investigated in \cite{BMRRT} as the orbit category of
$D^b(H)/\tau^{-1}[1]$, where $\tau$ is the Asuslander-Reiten
translation and [1] is the shift. The special case of Dynkin type
$A_n$ has been also studied in {\cite{BMR2}}. A cluster tilting
object $T$ in the cluster category is an object such that
$\Ext^1_{\mathcal{C}}(T,T) = 0$ and it is maximal with this
property.

The cluster tilted algebras were first introduced by Buan, Marsh and
Reiten in \cite{BMR2} as endomorphism ring of cluster tilting objects
over the cluster category. Since then, this class of algebras has
been extensively studied. Examples of works along these lines can be
found, for instance, in {\cite{BMR1}, \cite{BMR2}, \cite{BMR3},
\cite{BR}, \cite{BRS}, \cite{CCS}, \cite{CCS2} , \cite{KR}}

In particular, Assem, Br$\ddot{u}$stle and Schiffler have
shown in \cite{ABS} that cluster tilted algebras are trivial
extensions.  More precisely, they defined for an algebra $B$ with
${\rm gl dim} B$ at most two the relation extension of $B$ as the
algebra $\mathcal{R} (B) = B \semidirprod \Ext^2_ B({\rm D}B, B)$, where
${\rm D}B = \Hom _B( B, k)$. Then the authors proved that an
algebra $\mathcal{C}$ is a cluster tilted algebra if and only if
it is the relation extension of some tilted algebra $B$. This
connection was extended to iterated tilted algebras of global
dimension at most two in \cite{BFPPT}. In the mentioned work it
was shown that an algebra $\mathcal{C}$ is cluster tilted if and
only if there exists an iterated tilted algebra $B$ of ${\rm gl
dim}B \leq 2$ and a sequence of homomorphisms
$$B \rightarrow \mathcal{C} \overset{\pi} {\rightarrow}
\mathcal{R}(B) \rightarrow B$$ \noindent whose composition is the
identity map and the kernel of $\pi$ is contained in $\rm{rad}^2 C$. They
also proved that $C$ and $\mathcal{R}(B)$ have the same quivers.

In \cite{A}, Amiot has proved that for any algebra $B$ of global
dimension at most two, there exists a cluster category
associated with it such that $N$ is a cluster tilting object in this
category. The endomorphism ring of $A$ over that cluster category is the
tensor algebra over the $B$-$B$-bimodule $\Ext^{2}_B(DB,B)$. In the case that
$B$ is an iterated tilted algebra the mentioned endomorphism algebra
is a cluster tilted algebra.

The $m$-cluster category was introduced by Thomas in \cite{H} as the
orbit category $\mathcal C_m = D^b(H)/\tau^{-1}[m]$,$T$ is said to be an $m$-cluster
tilting object if $\Ext^i_{{\mathcal{C}_m(H)}}(\widetilde{T},\widetilde{T}) = 0$ for
$i = 1,\cdots, m$  and the number of isomorphism classes of
indecomposable summands of $T$ is equal to the number of
isomorphism classes of simple $H$-modules. The endomorphism algebra  $C_m$ of an
$m$-cluster tilting object over the $m$-cluster category is called
an $m$-cluster tilted algebra. The procedure introduced by Amiot in
\cite{A} can be extended to algebras of global dimension at most $m+1$, see \cite{IO}, and
$C_m$ is the tensor algebra over the $B$-$B$ bimodule
$\Ext^{m+1}_B(DB,B)$. In the case that $B$ is iterated tilted, this
endomorphism algebra is an $m$-cluster tilted algebra. Then the
$m$-cluster tilted algebras are tensor algebras. We are going to
show that, in particular cases, they can be seen as trivial
extensions.

 Our objective in this paper is to study the relation between iterated tilted
algebras of global dimension at most $m+1$ and $m$-cluster
tilted algebras.  For $B$, with global dimension at most $m+1$ the
$m$-relation extension of $B$ is the algebra $\mathcal{R}_m (B) = B
\semidirprod \Ext^{m+1}_ B({\rm D}B, B)$, where ${\rm D}B = \Hom _B( B, k)$.
Let $S_m$ be the standard fundamental domain of $\mathcal{C}_m$ and
$\tilde{X}$ be the class in the orbit category of an object $X$ in
$\Der(H)$. Following  \cite{R}, we recall that an object $T$ in $\Der(H)$ is a tilting complex if and only if ${\rm
Hom}_{\Der(H)}(T,T[i])=0$ for all $i \neq 0$ and $T$ has exactly
$n$ non-isomorphic summands, where $n$ is the number of vertices of $Q$. It follows from \cite{H} that, $B$ is an iterated tilted algebra if and only if $B = {\rm End}_{\Der(H)}(T)$ where $T$ is a
tilting complex.

The following is our first result.

{\bf Theorem 1.}  Let $T$ be a tilting complex in $\Der(H)$ which
belongs to $S_m$. Then ${\End}_{{\mathcal{C}}_m}(\widetilde{T})$ is an
$m$-cluster tilted algebra. Moreover if global dimension of $B = {\rm End}_{\Der(H)}(T)$ is at most $m+1$,
then ${\End}_{{\mathcal{C}}_m}(\widetilde{T})$ is isomorphic to
$\mathcal{R}_m(B)$.

Given an iterated tilted algebra $B = {\rm End}_{\Der(H)}(T)$ of global dimension
at most $m+1$, the corresponding
tilting complex $T$ can be spread in an arbitrary number of copies of mod $H$ inside the bounded
derived category of $H$. In order to get a tilting complex in the fundamental domain $S_m$,
we show that the rolling procedure introduced in \cite{BFPPT} can be generalized to a new
procedure $\rho_m$. We show that, iterating this procedure, we eventually get for some integer $h$
a tilting complex $\rho_m^h(T)$ in the fundamental domain, such that $T$ and $\rho_m^h(T)$
represent the same object in the $m$-cluster category. Moreover, the global dimension of ${\rm End}_{\Der(H)} \rho_m^h(T)$
is at most $m+1$. The following is our main theorem.

{\bf Theorem 2.}  Let $H=kQ$ a hereditary algebra. If $T$ is a tilting
complex in $\Der(H)$ such that $B = {\rm End}_{\Der(H)}(T)$ has global
dimension at most $m+1$ then

\vskip.1in a) $\widetilde{T}$ is an $m$-cluster tilting object in
the $m$-cluster category $\mathcal{C}_m$ and
$\mathcal{C}_m(B)={\rm End}_{\mathcal{C}_m}(\widetilde{T})$ is an
$m$-cluster tilted algebra.

\vskip.1in b) There exists a sequence of algebra homomorphisms

$$B \rightarrow \mathcal{C}_m(B) \overset{\pi} {\rightarrow}  \mathcal{R}_m(B)
\rightarrow B$$ \noindent whose composition is the identity map and the kernel is contained in $\rm{rad}^2 C$.
Moreover $C$ and $\mathcal{R}_m(B)$ have the same quiver.

\vskip.1in c) There exists an iterated tilted algebra $B' = \rm End_{\Der(H)} \rho^h (T)$ of type
$Q$ with ${\rm gl dim}B' \leq m+1$ such that:

$$\mathcal{C}_m(B) \simeq \mathcal{R}_m(B')\simeq
\mathcal{C}_m(B')$$

On the other hand, this kind of $m$-cluster tilted algebras  is in connection with
the higher relation extensions introduced by Assem, Gatica and Schiffler in \cite{AGS}.
More precisely they are quotients of higher relations extensions. Then, in some particular
cases, it is possible to describe the quivers with relations of  the $m$-cluster tilted algebras
induced by iterated tilted algebras.

A natural question to investigate is to know whether any $m$-cluster tilted algebra is induced by iterated
tilted algebras. That is, if any  $m$-cluster tilted algebra is a direct product of $m_i$-relation extensions of
iterated tilted algebras. In \cite{M}, Murphy has classified $m$-cluster tilted algebras of type $A_n$.
We observe that using this classification is not difficult to prove that the algebras in this class
are of the desired form.

In Section 1 we fix the notation and recall some well known facts about $m$-cluster categories.
In  Section 2 we prove Theorem 1 and we show with an example that the hypothesis that the tilting complex belongs to the fundamental domain is essential.
In Section 3 we define the $m$-rolling procedure. We show that given an arbitrary tilting complex whose endomorphism algebra has global dimension at most $m+1$    we eventually reach, by iteration of this procedure,  a tilting complex in the fundamental domain such that the global dimension of its endomorphism algebra is also bounded by $m+1$. Finally we prove our main Theorem.

\section{Preliminaries}

Throughout this paper let $Q$ be a finite connected quiver without
oriented cycles, and $k$ an algebraically closed field. Then $H =
kQ$ is a hereditary finite dimensional algebra. We denote by mod$H$
the category of  finitely generated right modules over $H$.

As usual, ${\rm ind}H$ denotes a full subcategory whose objects are a full
set of representatives of the isomorphism classes of
indecomposable $H$-modules.

We consider $\Der(H)= \Der({\rm mod}H)$  the  derived category of
bounded complexes of finitely generated $H$-modules. Recall that
in $\Der(H)$ Serre Duality holds, that is, $${\rm
Hom}_{\Der(H)}(X,\tau Y)\simeq D{\rm Hom}_{\Der(H)}(Y,X[1])$$
\noindent for $X$, $Y$ in $\Der(H)$, where $[1]$ is the shift functor and $\tau$ is the AR-translation in $D^b(H)$.

It follows from \cite{H} and \cite{R} that if $T$ is a tilting complex then there exists an
equivalence of triangulated categories $G: \Der(H)\rightarrow
\Der(B)$ such that $G(T)=B$ and $G(\tau T[1])= {\rm D}B$.
Following \cite{BFPPT} we denote :

\vskip.1in

\centerline {$P_{X,T}=G(X)$ and $I_{X,T}=G(\tau X[1])$.}
\vskip.1in

for any direct summand $X$ of
$T$

The cluster category associated with $H$ was defined and
investigated in \cite{BMRRT}. Later, for a positive integer $m$,
the $m$-cluster category was studied in \cite{K} (see
[22,25,28,8]). By
\cite{K} we know that ${\mathcal{C}_m(H)}$ is a triangulated
category.  The objects in the m-cluster category are the $F_m =
\tau^{-1}[m]$ orbits of objects in $\Der(H)$ and the morphisms are
given by $\Hom_{\mathcal{C}_m(H)}(\widetilde{X},\widetilde{Y}) =
\bigoplus_{i \in Z} \Hom_{\Der(H)}(X,F_m^iY)$ where $X$, $Y$ are
objects in $\Der(H)$ and $\widetilde{X}$, $\widetilde{Y}$ are
their respective $F_m$-orbits.

It is clear that $S_m = \bigcup _{i=0}^{m-1}\Ind H[i] \bigcup
H[m]$ is a fundamental domain for the action of $F_m$ in ${\rm
Ind}\,\Der(H)$. The set $S_m$ contains exactly one representative
from each $F_m$-orbit in ${\rm Ind }\,\Der(H)$.

A characterization of object $\widetilde{T}$ of ${\mathcal{C}_m(H)}$ as an
m-cluster tilting object is given by
$\Ext^i_{{\mathcal{C}_m(H)}}(\widetilde{T},\widetilde{T}) = 0$ for
$i = 1,\cdots, m$  and the number of isomorphism classes of
indecomposable summands of $T$ is equal to the number of
isomorphism classes of simple $H$-modules. From now on, we will
assume for simplicity that $\widetilde{T}$ is basic in the usual sense.

\section{Trivial extensions and {\bf m}-cluster tilted algebras}
In all that follows let $B$ be a connected finite dimensional
algebra over an algebraically closed field $k$ and $M$ a $B$-$B$-bimodule. Recall that the trivial
extension of $B$ by $M$ is the algebra $B \semidirprod  M$ with
underlying vector space $B \oplus M$, and multiplication given by
$(a,m)(a',m') = (aa',am' + ma')$, for any $a,\;a' \in B$ and
$m,\;m' \in M$. For further properties of trivial extensions we
refer the reader to \cite{FGR}.
Let $B$ be an algebra of global dimension at most $m+1$ and $DB=
{\rm Hom}_k(B,k)$. The trivial extension $\mathcal{R}_m(B)=B
\semidirprod  \Ext^{m+1}_B(DB,B)$ is called the $m$-relation extension of $B$.
This concept is the $m$-ified analogue of the notion of relation
extension introduced in \cite{ABS}. In this work the authors
proved that an algebra $\mathcal{C}$ is a cluster tilted algebra
if and only if it is the relation extension of some tilted algebra $B$.

We are going to show that there exists a connection between the $m$-relation extensions of iterated
tilted algebras of global dimension at most $m+1$ and  $m$-cluster tilted algebras.
In fact, we are going to prove that if the tilting complex $T$ belongs to the fundamental domain $S_m$, then the
  the $m$-relation extension of the iterated tilted algebra is an $m$-cluster tilted algebra.

We start with the following technical lemma, which will be needed in the sequel. When $m=1$ this result is proven in
  \cite{BMRRT}, Proposition 1.5(a).

\begin{lem}\label{I} Let $X$ and $Y$ be objects of $S_m
 = \bigcup_{i=0}^{m-1} \Ind H[i] \bigcup H[m]$.

 Then
 $\Hom_{\Der(H)}(F_m^iX,Y) = 0$ for all $i \neq -1,\;0$.
\end{lem}

\begin{proof} We have $\Hom_{\Der(H)}(F_m^{i}X,Y) =
\Hom_{\Der(H)}(\tau^{-i}X[im],Y)$. We prove the statement by
considering eight cases, according the integer $i\neq -1,0$ is positive or negative, and the objects $\tau^{-1}X, Y$ in $S_m$
 are in $\bigcup_{i=0}^{m-1} \Ind H[i]$ or in $ \bigcup H[m]$.

\vskip.2in (1) Let $i \geq 1$ and $\tau^{-i}X$, $Y$ be objects in
$\bigcup_{j=0}^{m-1} \Ind H[j]$. Then $Y = Z[r]$ and $\tau^{-i}X =
L[s]$ with $0 \leq r,\; s \leq m-1$. We get that

\noindent $\Hom_{\Der(H)}(F_m^{i}X,Y) =
\Hom_{\Der(H)}(\tau^{-i}X[im],Y) = \Hom_{\Der(H)}(L[s +
im],Z[r])$ $=0$, because $s+im > m$ and $r<m$

\vskip.2in (2) Let $i\geq 1$, $\tau^{-i}X = L[s]$, with $0 \leq s
\leq m-1$ and $Y = P[m]$, for some indecomposable projective
$H$-module $P$.

We get $\Hom_{\Der(H)}(F_m^{i}X,Y) =
\Hom_{\Der(H)}(\tau^{-i}X[im],Y) = \Hom_{\Der(H)}(L[s +im],P[m])=
\Hom_{\Der(H)}(L[(i-1)m+s],P)$. In the case  $i = 1$ we obtain that  $s \neq
0$. In fact, if $i = 1$ and $s = 0$ we have
$\Hom_{\Der(H)}(L,P) \not= 0$. Hence there exists a projective
$H$-module $P^{\prime}$ such that $ L = \tau^{-1}X = P^{\prime}$.
On the other hand, $\tau P^{\prime} = I[-1]$ with $I$ an injective module, so that
$X=I[-1]$. This  contradicts the fact that $X \in S_m$ and proves that $s\neq 0$ when $i=1$. The statement holds in this case because  $\Hom_{\Der(H)}(L[s],P) = 0$.

If $i>1$ then $(i-1)m \geq m$, so $\Hom_{\Der(H)}(L[(i-1)m+s],P) =
0$ .

\vskip.2in (3) Consider $i \geq 1$, $\tau^{-i}X = P[m]$ for all
indecomposable projective $H$-module $P$, and $Y = Z[r]$ with $0
\leq r \leq m-1 $. Then $(i+1)m > m>r$. It follows that

$$\Hom_{\Der(H)}(F_m^{i}X,Y) = \Hom_{\Der(H)}(\tau^{-i}X[im],Y) $$
$$=
\Hom_{\Der(H)}(P[(i+1)m],Z[r])= 0$$.

\vskip.2in (4) Suppose  that $i \geq 1$, $\tau^{-i}X = P[m]$ and
$Y = P'[m]$, where $P$, $P'$ are indecomposable projective
$H$-modules. Since $(i+1)m > m$ and

$$\Hom_{\Der(H)}(F_m^{i}X,Y) = \Hom_{\Der(H)}(\tau^{-i}X[im],Y)$$
$$ = \Hom_{\Der(H)}(P[(i+1)m],P'[m])$$, we obtain that there are no
nonzero morphism from $\tau^{-i}X[im]$ to $Y$.

\vskip.2in (5) Let $i \leq -2$ and let $\tau^{-i}X$, $Y$ be
objects in $\bigcup_{j=0}^{m-1}\Ind \, H[j]$. Then  $\tau^{-i}X =
L[s]$, $Y = Z[r]$, with $0 \leq r,s \leq m-1$, and so

{\small
$\Hom _{\Der(H)}(F_m^{i}X,Y) = \Hom_{\Der(H)}(\tau^{-i}X[im],Y) =
\Hom_{\Der(H)}(L[(s+im],Z[r])$. }Assume first that  $s \leq r$. 
Let
$k > 0$ such that $s +k = r$. Thus $\Hom_{\Der(H)}(L[(s+im],Z[r])
= \Hom_{\Der(H)}(L[(s+im],Z[s+k]) = \Hom_{\Der(H)}(L[(im],Z[k]) =$

$\Hom_{\Der(H)}(L[(im - k],Z)$ is equal to zero because $im - k <
-2m \leq -2$.

Now for $r < s$, we have $r+k = s$, with $k > 0$. It follows that
$im+k \leq -2m + m -1 = -m -1 \leq -2$. Therefore
$\Hom_{\Der(H)}(L[(s+im],Z[r]) = \Hom_{\Der(H)}(L[(r+k+im],Z[r]) =
\Hom_{\Der(H)}(L[(k+im],Z) = 0$.

\vskip.2in 
(6) Suppose that $i \leq -2$, $\tau^{-i}X = P[m]$, $P$
an indecomposable projective $H$-module, and $Y =Z[r]$ with $0
\leq r \leq m-1$. We have

\noindent
{\small
$\Hom_{\Der(H)}(F_m^{i}X,Y) = \Hom_{\Der(H)}(\tau^{-i}X[im],Y)=
\Hom_{\Der(H)}(P[(i+1)m],Z[r]) = \Hom_{\Der(H)}(P[(i+1)m-r],Z) =
0$}
, because $(i+1)m -r < 0$ and $P$ is projective.

\vskip.2in (7) If $i \leq -2$, $\tau^{-i}X = L[s]$ with $0 \leq s
\leq m-1$, and  $Y =P[m]$ for some indecomposable projective
$H$-module $P$, then $s + im -m = s + (i-1)m \leq -3m + s \leq -3.$
On the other hand $\Hom_{\Der(H)}(F_m^{i}X,Y) =
\Hom_{\Der(H)}(\tau^{-i}X[im],Y) = \Hom_{\Der(H)}(L[(s+im],P[m]) =
\Hom_{\Der(H)}(L[(s+im-m],P)$. Thus we obtain that there are no
nonzero morphisms from $\tau^{-i}X[im]$ to $Y$.

\vskip.2in (8) Finally, for $i \leq -2$, $\tau^{-i}X = P[m]$, $Y =P'[m]$,
with $P$, $P'$ indecomposable projective $H$-modules, we have that
$$\Hom_{\Der(H)}(F_m^{i}X,Y) = \Hom_{\Der(H)}(\tau^{-i}X[im],Y)$$Ç
$$= \Hom_{\Der(H)}(P[((i+1)m],P'[m])
= \Hom_{\Der(H)}(P[im],P') = 0$$
since $im \leq -2$.

\vskip.1in In all cases $\Hom_{\Der(H)}(F_m^{i}X,Y) =0$ for $i\neq -1,0$, so the proof of the lemma is complete.
\end{proof}

\vskip.2in \noindent In analogy with \cite{ABS} we have the
following.

\begin{rem}Let $X$ be an object in $\Der(H)$. We can consider the
$k$-vector space $\Hom_{\Der(H)}(X,F_mX)$ as an
End$_{\Der(H)}(X)$-bimodule defining $ufv = F_m(u)  f  v$ for $u$,
$v$ in End$_{\Der(H)}X$ and $f$ in $\Hom_{\Der(H)}(X, F_mX)$
\end{rem}

Next we show that the endomorphism algebra
End$_{\mathcal{C}_m}(\widetilde{X})$ is a trivial extension.  This
result is an m-ified analogue of Lemma 3.3 of \cite{ABS}.

\begin{prop}\label{C} Let $\widetilde{X}$ be an object in  ${\mathcal{C}}_m$
induced by an object $X$ in $S_m$. Then ${\rm
End}_{{\mathcal{C}}_m}(\widetilde{X}) ={\rm End}_{\Der(H)}X
\semidirprod  \Hom_{\Der(H)}(X,F_m X)$
\end{prop}

\begin{proof} We have
$\End_{{\mathcal{C}}_m}(\widetilde{X})= \bigoplus_{i \in Z}
\Hom_{\Der(H)}(X,F^i_mX)$ as $k$-vector spaces, and the product is
given by

\[(g_i)_{i \in Z}(f_j)_{j \in Z} = (\sum_{i+j=l}F^j_mg_if_j)_{l \in
Z}\]

Applying Lemma \ref{I}, it follows that

${\rm End}_{{\mathcal{C}}_m}(\widetilde{X}) = \Hom_{\Der(H)}(X,X)
\bigoplus \Hom_{\Der(H)}(X,F_mX)$

Let $f$, $g$ be elements in
$\End_{{\mathcal{C}}_m}(\widetilde{X})$. Then $f=(f_0,f_1)$ and
$g=(g_0,g_1)$ where $f_0$, $g_0$ are in $\Hom_{\Der(H)}(X,X)$ and
$f_1$, $g_1$ are in $\Hom_{\Der(H)}(X,F_mX)$. Since

$\Hom_{\Der(H)}(X,F^2_mX) = \Hom_{\Der(H)}(X,\tau^{-2}X[2]) = 0$, we obtain that $F_mg_1f_1$ $ = 0$, and therefore $gf
=(g_0f_0,F_mg_0f_1 + f_0g_1)$. Considering the bimodule structure
of $\Hom_{\Der(H)}(X,F_mX)$ given above, we conclude that
$\End_{{\mathcal{C}}_m}(\widetilde{X})$ is the trivial extension
of
 $\End_{\Der(H)}(X)$ by the bimodule $\Hom_{\Der(H)}(X,F_m X)$,
 as desired.
 \end{proof}

\begin{lem} \label{E} Let $T$ be a tilting complex in $\Der(H)$, and let $B =
\End_{\Der(H)}T$ an iterated tilted algebra. Then
$\Hom_{\Der(H)}(T,F_mT) = \Ext^{m+1}_B(DB,B)$.
\end{lem}

\begin{proof} For $m=1$ this statement is proven in \cite{ABS}, as part of the proof of their main result. It follows from
\cite{H} that

 $\Hom_{\Der(H)}(T,F_mT) \simeq \Hom_{\Der(B)}(B,F'_mB)$, where
 $F'_m = \tau^{-1}_{{\Der(B)}}[m]$ is the functor induced by
 $F_m = \tau^{-1}[m]$ in $\Der(B)$.

 We have 
 $$\Hom_{\Der(B)}(B,F'B)
 \simeq \Hom_{\Der(B)}(\tau B[1],B[m+1])
 \simeq \Hom_{\Der(B)}(DB,B[m+1])$$
 $$\simeq \Ext^{m+1}_B(DB,B)$$.
 \end{proof}

\vskip.2in Now, we are in position to state the following
essential result.

\begin{thm}\label{T} Let $T$ be a tilting complex in $\Der(H)$
which belongs to $S_m$. Then
${\End}_{{\mathcal{C}}_m}(\widetilde{T})$ is an $m$-cluster tilted
algebra. Moreover if global dimension of $B  = {\rm End}_{\Der(H)}(T)$ is at most  $m+1$, then
${\End}_{{\mathcal{C}}_m}(\widetilde{T})$ is isomorphic to
$\mathcal{R}_m(B)$.
\end{thm}

\begin{proof} It follows from [BRT, prop. 2.4.] that
$\widetilde{T}$ is an $m$-cluster tilting object in
$\mathcal{C}_m(H)$. Then ${\End}_{{\mathcal{C}}_m}(\widetilde{T})$
is an $m$-cluster tilted algebra. The second fact follows from
Proposition \ref{C} and Lemma \ref{E}.
\end{proof}

The following example shows that ${\End}_{{\mathcal{C}}_m}(\widetilde{T})$ is not necessarily isomorphic to the trivial extension $\mathcal{R}_m(B)$ when $T$ is not in the fundamental domain $S_m$.

\begin{ej}\label{Ex}
 Consider the tilting complex $T$  in $\Der(H)$, where the indecomposable summands $T_i$ of $T$ are indicated in the
following picture. In the picture the
Auslander-Reiten quiver $\Gamma$ of the derived category $\Der(H)$
is indicated; the arrows are going from left to right and are
drawn as lines to simplify the picture. The indecomposable summand
$T_i$ of the tilting complex $T=\oplus_{i=1}^7 T_i$ has been
indicated by the number $i$ inside a circle, that is, the symbol
\begin{picture}(10,8)\put(5,3){\mylabel{$i$}}\end{picture}.
Furthermore, $F_2^{-1}T_i$, resp. $F_2 T_i$, has been indicated by
the symbol
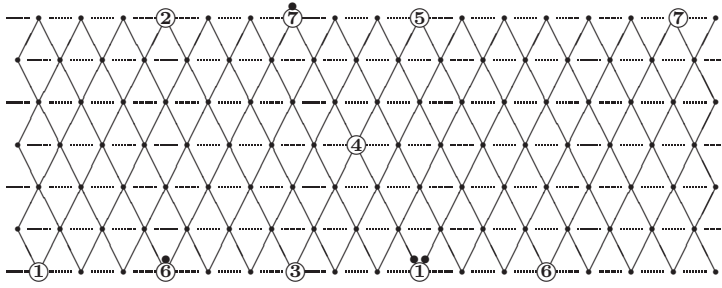
\begin{picture}(10,8)\put(5,3){\mylabeltwo{$i$}{$\bullet$}}\end{picture},
resp.
\begin{picture}(10,8)\put(5,3){\mylabeltwo{$i$}{$\bullet\bullet$}}\end{picture}.

\begin{center}
  \begin{picture}(352,116)
    \put(50,10){
      \put(0,80){\line(1,2){8}}
      \put(0,48){\line(1,2){24}}
      \put(0,16){\line(1,2){40}}
      \multiput(8,0)(16,0){14}{\line(1,2){48}}
      \put(232,0){\line(1,2){32}}
      \put(248,0){\line(1,2){16}}

      \put(0,16){\line(1,-2){8}}
      \put(0,48){\line(1,-2){24}}
      \put(0,80){\line(1,-2){40}}
      \multiput(8,96)(16,0){14}{\line(1,-2){48}}

      \put(248,96){\line(1,-2){16}}
      \put(232,96){\line(1,-2){32}}

      \multiput(0,0)(0,32){4}{
        \multiput(-8,0)(16,0){17}{
          \qbezier[7](4,0)(8,0)(12,0)
        }
      }

      \multiput(0,16)(0,32){3}{
        \multiput(0,0)(16,0){17}{
          \qbezier[7](4,0)(8,0)(12,0)
        }
      }

      \multiput(0,0)(0,32){4}{
        \multiput(8,0)(16,0){17}{\mypdott}
      }

      \multiput(0,0)(0,32){3}{
        \multiput(0,16)(16,0){17}{\mypdott}
      }
      \put(8,0){\mylabel{1}}
      \put(56,96){\mylabel{2}}
\put(56,0){\mylabeltwo{6}{$\bullet$}}
      \put(105,0){\mylabel{3}}
      \put(128,48){\mylabel{4}}
      \put(152,96){\mylabel{5}}
      \put(200,0){\mylabel{6}}
      \put(250,96){\mylabel{7}}
      \put(152,0){\mylabeltwo{1}{$\bullet\bullet$}}
      \put(104,96){\mylabeltwo{7}{$\bullet$}}
   }
  \end{picture}
\end{center}

The corresponding iterated tilted algebra $B= {\rm
End}_{\Der(H)}(T)$ is given by the quiver with relations

\begin{picture}(10,57)

   \put(100,27){
        \put(0,0){\RVCenter{\tiny $1$}}
        \put(30,0){\RVCenter{\tiny $2$}}
        \put(60,0){\RVCenter{\tiny $3$}}
        \put(90,0){\RVCenter{\tiny $4$}}
        \put(120,0){\RVCenter{\tiny $5$}}
        \put(150,0){\RVCenter{\tiny $6$}}
        \put(180,0){\RVCenter{\tiny $7$}}

\multiput(3,0)(30,0){6}{\vector(1,0){20}}

            \qbezier[20](10,3)(26,20)(42,3)
            \qbezier[20](40,3)(56,20)(72,3)
            \qbezier[20](100,3)(116,20)(132,3)
            \qbezier[20](130,3)(146,20)(162,3)
    }
\end{picture}
Clearly ${\rm gl dim}B =3$. Let $m=2$ and $ \mathcal{C}_2$ be the
$2$-cluster category of $H$. Then $\mathcal{C}_2(B)$ is given by
the quiver

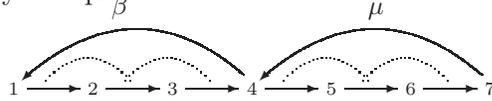
\begin{picture}(10,57)
 \put(100,27){
        \put(0,0){\RVCenter{\tiny $1$}}
        \put(30,0){\RVCenter{\tiny $2$}}
        \put(60,0){\RVCenter{\tiny $3$}}
        \put(90,0){\RVCenter{\tiny $4$}}
        \put(120,0){\RVCenter{\tiny $5$}}
        \put(150,0){\RVCenter{\tiny $6$}}
        \put(180,0){\RVCenter{\tiny $7$}}

\multiput(3,0)(30,0){6}{\vector(1,0){20}}

            \qbezier[20](10,3)(26,20)(42,3)
            \qbezier[20](40,3)(56,20)(72,3)
            \qbezier[20](100,3)(116,20)(132,3)
            \qbezier[20](130,3)(146,20)(162,3)

 \put(35,28){\small $\beta$}
\put(132,28){\small $\mu$}
            \qbezier(85,5)(45,40)(3,5)
                  \put(1,3){\vector(-1,-1){0.1}}
            \qbezier(175,5)(135,40)(93,5)
                  \put(91,3){\vector(-1,-1){0.1}}
    }

\end{picture}

\noindent The relations of $\mathcal{C}_2(B)$ are given by the composition of any two consecutive arrows in each cycle.

$\mathcal C_2(B) = {\End}_{{\mathcal{C}}_2}(\widetilde{T}) = B \oplus \Ext^3_B (B,DB) \oplus \Ext^3_B (B,DB)  \otimes \Ext^3_B (B,DB).$

 Since there exists a sectional path in $\Der(H)$ between the vertices $F_2^{-1}T_7$ and $F_2 T_1$ passing through the vertex $T_4$, we have that $\mu \beta \not= 0$.  It follows that $\Ext^3_B (B,DB)  \otimes \Ext^3_B (B,DB) \not= 0$ and
$\mathcal{C}_2(B)$ is not isomorphic to the trivial extension
$\mathcal{R}_2(B)$. We observe that $\mathcal{C}_2(B)$ and $\mathcal{R}_2(B)$ have the same quivers.

\end{ej}

\section{The $m$-rolling of tilting complexes}

A very useful tool introduced  in \cite{BFPPT} to study tilting complexes and iterated tilted algebras of global dimension at most two, is a procedure called the
rolling of tilting complexes. In this section we want to show how
the mentioned  procedure can be carried on in our context. Now, we extend
most of the notions and facts from (\cite{BFPPT}, section 3).

We define for each tilting complex  $T$
 a new tilting complex
$\rho_m (T)$ such that $T \simeq \rho_m(T)$ in the $m$-cluster
category $\mathcal{C}_m$, and such that iterating this procedure we eventually obtain a tilting complex in the fundamental domain of the cluster category.

Let $Q$ be a Dynkin quiver and $T$ a tilting complex of $\Der(k
Q)$.

Since $T=\bigoplus_{i=1}^n T_i$ has only finitely many summands we
can easily find a section $\Sigma=\{\Sigma_1,\ldots,\Sigma_n\}$
such that $T\leq \Sigma$, that is, $T_i\leq \Sigma_j$ for all $i$
and $j$. If $\Sigma_j$ is maximal in $\Sigma$ and
$\Sigma_j\not\in\{T_1,\ldots,T_n\}$ then
$\Sigma'=\Sigma\setminus\{\Sigma_j\}\cup\{\tau\Sigma_j\}$ is also
a section satisfying $T\leq \Sigma'$. After finitely many steps we
get a section $\Sigma(T)$ with $T\leq \Sigma(T)$ and all maximal
elements in $\Sigma(T)$ belong to add$T$. Notice that the section
$\Sigma(T)$ is uniquely defined by $T$.

\begin{defn}[m-Rolling of tilting complex, the Dynkin case]
  \label{def:rolling-D}
With the previous notations, let $X$ be the sum of those summands
of $T$ which belong to $\Sigma(T)$ and $T'$ a complement of $X$ in
$T$. Then define the \emph{$m$-rolling} of $T$ to be
$\rho_m(T)=T'\oplus F_m^{-1} X$.
\end{defn}

Now consider the case where $Q$ is not Dynkin.
If $Q$ is not Dynkin then the structure of the Auslander-Reiten quiver $\Gamma$ of $\Der(H)$ is
completely different.  Denote by $\Pp$, (resp.  $\Ii$) the
preprojective (resp. preinjective) component of the Auslander-Reiten quiver of $H$ and by $\Rr$ the full
subcategory of $\mod H$ given by the regular components. For each $r\in \ZZ$ the regular
part $\Rr$ gives rise to $\Rr[r]$, given by the complexes $X\in  \Der(H)$ concentrated
in degree $r$ with $X_r\in\Rr$. Moreover, for each $r\in\ZZ$ there is a
transjective component $\Ii[r-1]\vee\Pp[r]$ of $\Gamma$ which we
shall denote by $\Rr[r-\frac{1}{2}]$ and each component of
$\Gamma$ is contained in $\Rr[r]$ for some half-integer $r$.  The notation has the advantage that the different
parts are ordered in the sense that $\Hom(\Rr[a],\Rr[b])=0$ for any
two half-integers $a>b$. Also note that $\Hom(\Rr[a],\Rr[b])=0$ if
$a<b-1$.

We know
 that $\Der(k Q)$ is composed by the parts
$\mathcal{R}[r]$ for $r \in \mathbb{Z}/2$ where $\mathcal{R}[r]$
denotes the regular (resp. transjective) part if $r$ is an integer
(resp. not an integer). Now, write
$T=\bigoplus_{a\in\mathbb{Z}/2}T_{\mathcal{R}[a]}$, where
$T_{\mathcal{R}[a]}\in\mathcal{R}[a]$.

\begin{defn}[$m$-Rolling of tilting complex, the non-Dynkin case]
  \label{def:rolling-nD}
With the previous notation let $s$ be the largest half-integer
such that $T_{\mathcal{R}[s]}$ is non-zero. Then define
$X=T_{\mathcal{R}[s]}$ and $T'$ to be the complement of $X$ in
$T$. Define the \emph{$m$-rolling} of $T$ to be $\rho_m(T)=T'\oplus
F_m^{-1} X$.
\end{defn}

\begin{rem}\label{hom XT}If $T=T'\oplus X$ is a tilting complex in $\Der(H)$ and
  $\rho_m(T)=T'\oplus F_m^{-1} X$ then we have $\Hom_{\Der(H)}(X,T')=0$.
\end{rem}

\begin{defn}[$m$-Rolling of iterated tilted algebras]
  Let $B$ be an iterated tilted algebra. Then define $\rho_m(B)$ to be
  the endomorphism algebra $\End_{\Der(H)}(\rho_m(T))$, where $H$ is a hereditary
  algebra with $\Der(B)\simeq \Der(H)$ and $T$ a tilting complex in
  $\Der(H)$ with $B=\End _{\Der(H)}(T)$.
\end{defn}

Notice that $\rho_m(B)$ does not depend on the choice of $H$ or
$T$.  See \cite{BFPPT}, Section 3.2.

The proofs of the following results are similar to those in
\cite{BFPPT} Section 3.3, so we shall omit them.

\begin{lem}\label{lem:crit}
  We consider $T = T^{\prime} \oplus X$  a tilting complex in $\Der(H)$ such
  that $\Hom_{\Der(H)}(X,T') = 0$ and let $B=\End_{\Der(H)} (T)$.
  Then $\overline{T} = T^{\prime} \oplus F_m^{-1}X$ is a tilting complex
  if and only if $\Hom_{\Der(H)} (F_m^{-1}X,T'[j])=0$ for all $j\neq 0$
  if and only if $\Ext^j_B(I_{X,T},P_{T',T}) = 0$ for each
  $j\neq m+1$ .
\end{lem}

\begin{lem}\label{lem:crit-strong}
We consider $T=T'\oplus X$  a tilting complex in $\Der(H)$ such that
$\Hom_{\Der(H)}(X,T')=0$ and let $B=\End_{\Der(H)}(T)$. If $\gldim
B\leq m+1$, then $\rho_m T=T'\oplus F_m^{-1}X$ is a tilting
complex in $\Der(H)$ if and only if $\Hom_{\Der(H)}(\tau
X,T'[k])=0$ for $k=0,-1,-2,...,-(m+1)$.
\end{lem}

{\begin{lem}\label{lem:nuevo} Let $Q$ be a Dynkin quiver and $T$ a
tilting complex in $D^b(H)$. Then $\rho_m(T) < \tau(\Sigma(T))$.
\end{lem}

}

\begin{prop}\label{prop:rho-T-tilting}
  Let $T$ be a tilting complex {in $\Der(H)$} such that global dimension of $\End_{\Der(H)}(T)$
  is at most $m+1$. Then $\rho_m(T)$ is again a tilting complex.
\end{prop}

The next result is analogous to Proposition 3.11 in
\cite{BFPPT}. We include the proof for the convenience of the
reader.

\begin{prop}\label{prop:gldim}
  Let $B$ be an iterated tilted algebra. If $\gldim B\leq m+1$ then
  $\gldim \rho_m(B)\leq m+1$.
\end{prop}

\begin{proof}
 Let $H$ be a hereditary algebra and $T=T'\oplus X$ a tilting complex
  in $\Der(H)$ such that $B=\End_{\Der(H)}(T)$ and $\rho_m(T)=T'\oplus
  F^{-1}_m X$. Then we have $\Hom_{\Der(H)}(X,T')=0$ by Remark
  \ref{hom XT}
  and by Proposition \ref{prop:rho-T-tilting} the complex $\rho_m(T)$ is
  a tilting complex in $\Der(H)$. To shorten notations we set
  $\widetilde{T}=\rho_m(T)$ and $\widetilde{B}=\rho_m(B)$.  We shall prove that
  $\Ext^j_{\widetilde{B}}({\rm D} \widetilde{B},\widetilde{B})=0$ for all $j \geq m+2$.  Since $\widetilde{T}$ is a
  tilting complex, we can show this by proving that
  $\Hom_{\Der(H)}(\tau \widetilde{T}[1],\widetilde{T}[j])$ is zero for $j\geq m+1$.

  First note that
  \begin{equation}\label{eq:star}
  \Hom_{\Der(H)}(\tau T[1], T[i])=0 \quad \text{ for all $i\neq 0,1,2,...,m+1$},
  \end{equation}
  since $\Hom_{\Der(H)}(\tau T[1], T[i]){\simeq}\Ext^i_B(D B,B)$.

  Therefore $\Hom_{\Der(H)}(\tau F^{-1}_m X[1],F^{-1}_m X[j])=\Hom_{\Der(H)}(\tau X,X[j-1])=0$  and $\Hom_{\Der(H)}(\tau T'[1],T'[j])=0$ for $j\geq m+2$.
  Also, $$\Hom_{\Der(H)}(\tau T'[1],F^{-1}_m X[j])=\Hom_{\Der(H)}(T'[1],X[j-m])$$, which is
  zero for all $j\neq m+2$  since $T$ is a tilting complex.

  Hence, it remains to see that $\Hom_{\Der(H)}(\tau^2 X[1],T'[j])=0$ for $j\leq
  2m+1$.  The minimal projective resolution of $I_{X,T}$ in ${\rm mod} B$
  $$
  0 \rightarrow P_{m+1}\rightarrow P_m \rightarrow P_{m-1} \overset{\varphi_{m-1}}{\rightarrow}
  \cdots \overset{\varphi_{2}}{\rightarrow} P_1 \overset{\varphi_{1}}{\rightarrow}
  P_0 \overset{\varphi_{0}}{\rightarrow} I_{X,T}\rightarrow 0
  $$

 \noindent gives rise to  exact triangles $$\Delta_1 \colon K_1 \rightarrow P_0\rightarrow
  I_{X,T}\rightarrow K_1[1]$$  $$\Delta_i\colon K_i \rightarrow P_{i-1}\rightarrow K_{i-1} {\rightarrow
    K_{i}[1]}$$   where $K_i$ denotes the kernel of $\varphi_i$ for $i=0,\cdots,m-1$
    and $$\Delta_{m+1} \colon P_{m+1} \rightarrow P_{m}\rightarrow
  K_m \rightarrow P_{m+1}[1]$$.

  Apply first the inverse of the equivalence
  $G \colon \Der(H)\rightarrow\Der(B)$ and then $\tau$, to obtain exact
  triangles of the form $S_1 \rightarrow \tau T_{0}\rightarrow \tau^2 X[1]\rightarrow
  S_1[1]$, $S_i \rightarrow \tau T_{i-1}\rightarrow S_{i-1}\rightarrow
  S_i[1]$ for $i=2,\cdots,m$ and
  $\tau T_{m+1}\rightarrow \tau T_m \rightarrow S_m \rightarrow \tau T_{m+1}[1]$
  with $S_i=\tau G^{-1}(K_i)$ and some $T_0, \cdots, T_{m+1} \in {\rm add }T$. To these
  triangles apply
  the homological functor $\Hom_{\Der(H)}(-,T'[j])$ to get exact sequences
  $$(\tau T_0[1], T'[j])\rightarrow (S_1[1],T'[j])\rightarrow (\tau^2
X[1],T'[j])\rightarrow (\tau
    T_0,T'[j])$$
    \noindent  for $i=2,\cdots, m$

$$  (\tau {T_{i-1}[i]},T'[j])\rightarrow (S_i[i],T'[j])\rightarrow (\tau {S_{i-1}[i-1]},T'[j])\rightarrow
    \tau {T_{i-1}[i-1]},T'[j]) $$
    \noindent and
$$(\tau {T_{m}[m+1]},T'[j])\rightarrow
(T_{m+1}[m+1],T'[j])\rightarrow (\tau {S_{m}[m]},T'[j])\rightarrow
    \tau {T_{m}[m]},T'[j])$$

 \noindent where we abbreviated $(Y,Z)=\Hom_{\Der(H)}(Y,Z)$.  By \eqref{eq:star}, the end
  terms of the previous sequences
  are zero for {$j>2m+1$} and hence we get $$\Hom _{\Der(H)}(\tau^2
  X[1],T'[j])\simeq \Hom _{\Der(H)}(S_m[m],T'[j])=0$$ for {$j>2m+1$}, which is what we
  wanted to prove.
\end{proof}

Using ideas similar to those in section 3.5 in \cite{BFPPT} we can
state the following Theorem.

\begin{thm}\label{T2} Let $B = {\rm End}_{\Der(H)}(T)$ be an iterated
tilted algebra of type $Q$ with ${\rm gl dim}B \leq m+1$. Then for
sufficiently large $h$ the tilting complex $\rho_m^h(T)$ is in
$S_m$.
\end{thm}


\section{The main result}

We may now state our main theorem.

\begin{thm}\label{MT}
Let $H=kQ$ a hereditary algebra. If $T$ is a tilting complex in
$\Der(H)$ such that $B = {\rm End}_{\Der(H)}(T)$ has global
dimension at most $m+1$, then

\vskip.1in a) $\widetilde{T}$ is an $m$-cluster tilting object in
the $m$-cluster category $\mathcal{C}_m$ and
$\mathcal{C}_m(B)={\rm End}_{\mathcal{C}_m}(\widetilde{T})$ is a
$m$-cluster tilted algebra.

\vskip.1in b) There exists a sequence of algebra homomorphisms

$$B \rightarrow \mathcal{C}_m(B) \overset{\pi} {\rightarrow}  \mathcal{R}_m(B)
\rightarrow B$$
\noindent whose composition is the identity map and the kernel is contained in $\rm{rad}^2 C$.
Moreover $C$ and $\mathcal{R}_m(B)$ have the same quiver..

\vskip.1in c) There exists an iterated tilted algebra $B' = \rm End_{\Der(H)} \rho^h (T)$ of type
$Q$ with ${\rm gl dim}B' \leq m+1$ such that:

$$\mathcal{C}_m(B) \simeq \mathcal{R}_m(B')\simeq
\mathcal{C}_m(B')$$
\end{thm}

\begin{proof}
Let $H=kQ$ be a hereditary algebra. By Theorem \ref{T2} we have
that there exists a number $h$ such that $\rho_m^h(T)$ is in
$\mathcal{S}_m$. It follows from Theorem \ref{T2} that the object
$\rho_m^h(T)$ defines an $m$-cluster tilting object in the
$m$-cluster category $\mathcal{C}_m$. Hence ${\rm
End}_{\mathcal{C}_m}(\rho_m^h(T))$ is an $m$-cluster tilted
algebra. Since $\rho_m^h(T) \simeq T$ in $\mathcal{C}_m$ we get
that $\mathcal{C}_m(B) = {\rm End}_{\mathcal{C}_m}(T)$ is an
$m$-cluster tilted algebra. This proves a).

The proof of b) goes through exactly as in \cite{BFPPT}.

For c) note that by iterating the rolling procedure we show in
Theorem \ref{T2} that there exists a number $h$ such that
$\rho_m^h(T)$ is a tilting complex in $\mathcal{S}_m$. Then the
algebra $B' = {\rm End}_{\Der(H)}(\rho_m^h(T))$ is an iterated
tilted algebra such that ${\rm gl dim}B' \leq m+1$. By Theorem
\ref{T} we have that $\mathcal{C}_m(B')= {\rm
End}_{\mathcal{C}_m}(\rho_m^h(T)) \simeq \mathcal{R}_m(B').$ Since
$\rho_m^h(T) \simeq T$ the result follows.
\end{proof}

We illustrate the former result by an example.

\begin{ej} We come back to Example \ref{Ex}. Consider again the tilting complex $T$

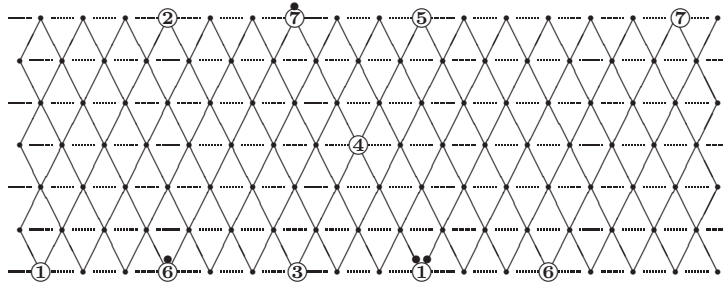
\begin{center}
  \begin{picture}(352,116)
    \put(50,10){
      \put(0,80){\line(1,2){8}}
      \put(0,48){\line(1,2){24}}
      \put(0,16){\line(1,2){40}}
      \multiput(8,0)(16,0){14}{\line(1,2){48}}
      \put(232,0){\line(1,2){32}}
      \put(248,0){\line(1,2){16}}

      \put(0,16){\line(1,-2){8}}
      \put(0,48){\line(1,-2){24}}
      \put(0,80){\line(1,-2){40}}
      \multiput(8,96)(16,0){14}{\line(1,-2){48}}

      \put(248,96){\line(1,-2){16}}
      \put(232,96){\line(1,-2){32}}

      \multiput(0,0)(0,32){4}{
        \multiput(-8,0)(16,0){17}{
          \qbezier[7](4,0)(8,0)(12,0)
        }
      }

      \multiput(0,16)(0,32){3}{
        \multiput(0,0)(16,0){17}{
          \qbezier[7](4,0)(8,0)(12,0)
        }
      }

      \multiput(0,0)(0,32){4}{
        \multiput(8,0)(16,0){17}{\mypdott}
      }

      \multiput(0,0)(0,32){3}{
        \multiput(0,16)(16,0){17}{\mypdott}
      }
      \put(8,0){\mylabel{1}}
      \put(56,96){\mylabel{2}}
\put(56,0){\mylabeltwo{6}{$\bullet$}}
      \put(105,0){\mylabel{3}}
      \put(128,48){\mylabel{4}}
      \put(152,96){\mylabel{5}}
      \put(200,0){\mylabel{6}}
      \put(250,96){\mylabel{7}}
      \put(152,0){\mylabeltwo{1}{$\bullet\bullet$}}
      \put(104,96){\mylabeltwo{7}{$\bullet$}}
   }
  \end{picture}
\end{center}

The corresponding iterated tilted algebra $B= {\rm
End}_{\Der(H)}(T)$ is given by the quiver with relations

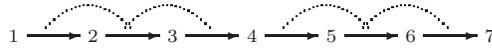
\begin{picture}(10,57)

   \put(100,27){
        \put(0,0){\RVCenter{\tiny $1$}}
        \put(30,0){\RVCenter{\tiny $2$}}
        \put(60,0){\RVCenter{\tiny $3$}}
        \put(90,0){\RVCenter{\tiny $4$}}
        \put(120,0){\RVCenter{\tiny $5$}}
        \put(150,0){\RVCenter{\tiny $6$}}
        \put(180,0){\RVCenter{\tiny $7$}}

\multiput(3,0)(30,0){6}{\vector(1,0){20}}

            \qbezier[20](10,3)(26,20)(42,3)
            \qbezier[20](40,3)(56,20)(72,3)
            \qbezier[20](100,3)(116,20)(132,3)
            \qbezier[20](130,3)(146,20)(162,3)
    }
\end{picture}
Consider $\mathcal{C}_2(B)$ given by
the following quiver where the relations are given by the composition of any two consecutive arrows in each cycle

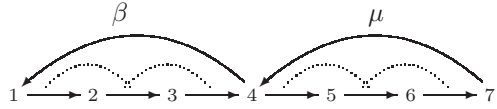
\begin{picture}(10,57)
 \put(100,27){
        \put(0,0){\RVCenter{\tiny $1$}}
        \put(30,0){\RVCenter{\tiny $2$}}
        \put(60,0){\RVCenter{\tiny $3$}}
        \put(90,0){\RVCenter{\tiny $4$}}
        \put(120,0){\RVCenter{\tiny $5$}}
        \put(150,0){\RVCenter{\tiny $6$}}
        \put(180,0){\RVCenter{\tiny $7$}}

\multiput(3,0)(30,0){6}{\vector(1,0){20}}

            \qbezier[20](10,3)(26,20)(42,3)
            \qbezier[20](40,3)(56,20)(72,3)
            \qbezier[20](100,3)(116,20)(132,3)
            \qbezier[20](130,3)(146,20)(162,3)

 \put(35,28){\small $\beta$}
\put(132,28){\small $\mu$}
            \qbezier(85,5)(45,40)(3,5)
                  \put(1,3){\vector(-1,-1){0.1}}
            \qbezier(175,5)(135,40)(93,5)
                  \put(91,3){\vector(-1,-1){0.1}}
    }

\end{picture}

\noindent We have seen that
$\mathcal{C}_2(B)$ is not isomorphic to the trivial extension
$\mathcal{R}_2(B)$.

Applying the $m$-rolling procedure we get that $T^*= T_1 \oplus T_2 \oplus T_3 \oplus T_4 \oplus T_5
\oplus F_2^{-1}T_6 \oplus F_2^{-1}T_7$ is a tilting complex such
that the indecomposable summands of it are in the fundamental
domain $\mathcal{S}_2$. The endomorphism algebra $B'={\rm
End}_{\Der(H)}(T^*)$ is given by the following quiver with relations

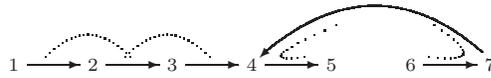
\begin{picture}(10,57)

   \put(100,27){
        \put(0,0){\RVCenter{\tiny $1$}}
        \put(30,0){\RVCenter{\tiny $2$}}
        \put(60,0){\RVCenter{\tiny $3$}}
        \put(90,0){\RVCenter{\tiny $4$}}
        \put(120,0){\RVCenter{\tiny $5$}}
        \put(150,0){\RVCenter{\tiny $6$}}
        \put(180,0){\RVCenter{\tiny $7$}}

\multiput(3,0)(30,0){4}{\vector(1,0){20}}
                  \put(153,0){\vector(1,0){20}}
            \qbezier[20](10,3)(26,20)(42,3)
            \qbezier[20](40,3)(56,20)(72,3)

            \qbezier[15](155,3)(180,0)(157,15)
            \qbezier[15](108,3)(85,0)(120,15)

            \qbezier(175,5)(135,40)(93,5)
                  \put(91,3){\vector(-1,-1){0.1}}    }
\end{picture}
By computing the $2$-cluster tilted algebra ${\rm
End}_{\mathcal{C}_2}(T^*)$ we get that $${\rm
End}_{\mathcal{C}_2}(T^*)=\mathcal{R}_2(B') = B' \semidirprod {\rm
Ext}^3_{B'}({\rm D}B',B') \simeq \mathcal{C}_2(B).$$

\begin{picture}(10,57)

   \put(100,27){
        \put(0,0){\RVCenter{\tiny $1$}}
        \put(30,0){\RVCenter{\tiny $2$}}
        \put(60,0){\RVCenter{\tiny $3$}}
        \put(90,0){\RVCenter{\tiny $4$}}
        \put(120,0){\RVCenter{\tiny $5$}}
        \put(150,0){\RVCenter{\tiny $6$}}
        \put(180,0){\RVCenter{\tiny $7$}}

\multiput(3,0)(30,0){6}{\vector(1,0){20}}

            \qbezier[20](10,3)(26,20)(42,3)
            \qbezier[20](40,3)(56,20)(72,3)

            \qbezier(85,5)(45,40)(3,5)
                  \put(1,3){\vector(-1,-1){0.1}}
            \qbezier(175,5)(135,40)(93,5)
                  \put(91,3){\vector(-1,-1){0.1}}

            \qbezier[15](155,3)(180,0)(157,15)
            \qbezier[15](108,3)(85,0)(120,15)    }
\end{picture}
\end{ej}
\begin{rem}
 In \cite{M}, Murphy has proved that $m$-cluster tilted algebras of type $A_n$ are gentle and that if the quiver of an m-cluster tilted algebra of type $A_n$ contains cycles they must be of length $m+ 2$ and must have full relations,
that is the composition of any two consecutive arrows in the cycle must be a relation. We observe that every connected component of an $m$-cluster tilted algebra of type $A_n$ is induced by a tilting complex. In fact,  we can choose the summands of the tilting complex in a similar way as we did in the example above. The ones which give rise to the cycles in analogous way, and for the hereditary parts we choose the summands over a complete slice. Then, any connected component of an $m$-cluster tilted algebra of type $A_n$ is an $m_i$-relation extension of an iterated tilted algebra of the same type. We conjecture that any $m$-cluster tilted algebra is a direct product of $m_i$-relation extensions of iterated tilted algebras.
\end{rem}

\end{document}